\documentclass[12pt]{amsart}                                 
\usepackage{amscd,amssymb,url,graphicx,hyperref,color}
\usepackage[all]{xy}

\newcommand{\onep}{{1'}}

\newtheorem{thm}{Theorem}
\newtheorem{lemma}[thm]{Lemma}
\theoremstyle{definition}
\newtheorem{ex}[thm]{Example}
\newcommand{\sections}{\renewcommand{\thethm}{\thesection.\arabic{thm}}
           \setcounter{thm}{0}}
\newcommand{\nosubsections}{\renewcommand{\thethm}{\thesection.\arabic{thm}}
           \setcounter{thm}{0}}
\newcommand{\co}{\colon}
\newcommand{\bD}{\mathbf{D}}
\newcommand{\bbR}{\mathbb{R}}
\newcommand{\bbZ}{\mathbb{Z}}
\newcommand{\za}{\alpha}
\newcommand{\zb}{\beta}
\newcommand{\zd}{\delta}
\newcommand{\zf}{\phi}
\newcommand{\zg}{\gamma}
\newcommand{\zj}{\psi}
\newcommand{\zl}{\lambda}
\newcommand{\zp}{\pi}

\newcommand{\zw}{\omega}
\newcommand{\zF}{\Phi}
\newcommand{\zG}{\Gamma}
\newcommand{\zJ}{\Psi}
\newcommand{\zL}{\Lambda}
\newcommand{\zt}{\tau}

\begin{document}
\title{Presentations of NET maps}

\begin{author}{William Floyd}
\address{Department of Mathematics\\ Virginia Tech\\
Blacksburg, VA 24061\\ USA}
\email{floyd@math.vt.edu}
\urladdr{http://www.math.vt.edu/people/floyd}
\end{author}

\begin{author}{Walter Parry}
\email{walter.parry@emich.edu}
\end{author}

\begin{author}{Kevin M. Pilgrim}
\address{Department of Mathematics, Indiana University, Bloomington, 
IN 47405, USA}
\email{pilgrim@indiana.edu}
\end{author}

\date{\today}

\begin{abstract} A branched covering $f: S^2 \to S^2$ is a {\em nearly
Euclidean Thurston} (NET) map if each critical point is simple and its
postcritical set has exactly four points. We show that up to
equivalence, each NET map admits a normal form in terms of simple
affine data. This data can then be used as input for algorithms
developed for the computation of fundamental invariants, now
systematically tabulated in a large census.
\end{abstract}



\maketitle

\tableofcontents

\sections

\section{Introduction }\label{sec:intro}\nosubsections

This paper is part of our program to thoroughly investigate nearly
Euclidean Thurston (NET) maps, introduced in \cite{cfpp}.  A Thurston
map $f: (S^2, P_f) \to (S^2, P_f)$ with postcritical set $P_f$ is
NET if each critical point is simple and $\#P_f=4$.  The set of NET
maps contains the exceptional set of {\em Euclidean NET maps} as a
proper subset.  Euclidean NET maps are the usual exceptional set of
maps whose canonical minimal associated orbifolds in the sense of
\cite{DH} are Euclidean and have signature $(2,2,2,2)$. In contrast,
typical NET maps have hyperbolic orbifolds.

This paper addresses the problem of succinctly describing a NET map in
a computationally effective manner.  Such a description is achieved by
what we call a NET map presentation diagram, a graphical
representation of what we call a NET map presentation.  In Section
\ref{sec:diagram2map}, we show that each of these simple diagrams
$\mathbf{D}$ drawn on standard graph paper, like that shown in Figure
\ref{fig:rabbitpren}, determines a NET map.
\begin{figure}
\centerline{\includegraphics{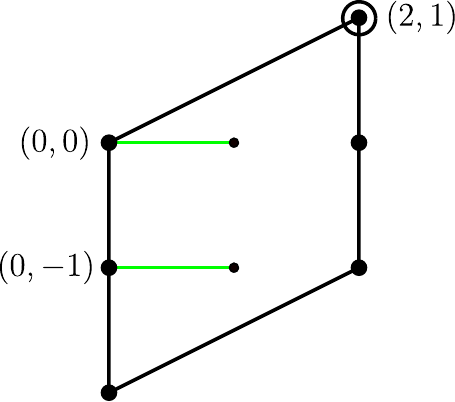}}
\caption{A NET map presentation diagram for Douady's rabbit quadratic 
polynomial}
\label{fig:rabbitpren}
  \end{figure}
Our first main result shows that every NET map arises in this way:

\begin{thm}\label{thm:pren}  Every NET map is Thurston
equivalent to one given by a NET map presentation diagram.
\end{thm}

Though far from unique, such presentation diagrams are nonetheless
useful for two reasons.

First, they give simple input data for algorithms for the computation of
fundamental invariants.  Parry, with initial assistance from Floyd,
wrote and continues to improve a computer program {\tt NETmap} which
takes as input NET map presentations and outputs a wealth of
information about the NET maps.  Much of our interest in NET maps lies
in their computational tractability.

Second, there is a correspondence between twists of NET map
presentation diagrams and twists of NET maps.  The set of NET map
presentation diagrams $\mathbf{D}$ admits a natural free action by the
monoid $\text{Mat}^+(2,\bbZ)$ of all $2\times 2$ matrices of integers
with positive determinant; the action of $A\in \text{Mat}^+(2,\bbZ)$
is to simply transform the diagram by $A$.  If we ``projectivize''
presentation diagrams by identifying $\bD$ with $-\bD$, then we obtain
a free action of $\text{PMat}^+(2,\bbZ)=\text{Mat}^+(2,\bbZ)/\{\pm
1\}$ on projective classes $[\bD]=\{\bD,-\bD\}$ of presentation
diagrams.  In particular, we obtain a free action of
$\text{PSL}(2,\bbZ)$ on projective classes of presentation diagrams.
This action has an interpretation as functional post-composition,
which we now describe. (In what follows, the subscript ``2'' serves to
later distinguish the constructions from closely related ones.)  Let
$\Gamma_2=\{x \mapsto \pm x + 2\lambda: \lambda \in \mathbb{Z}^2\}$,
let $S^2_2=\mathbb{R}^2/\Gamma_2$, let $\pi_2: \mathbb{R}^2 \to
S^2_2$ be the natural projection and let $P_2=\pi_2(\Lambda_2)\subset
S^2_2$.  The group $\text{PSL}(2,\bbZ)$ is naturally identified with
the stabilizer of $\zp_2(0)$ in the mapping class group of $(S_2^2,P_2)$.  Similarly, it follows from Section~\ref{sec:euc} that
$\text{PMat}^+(2,\bbZ)\setminus \text{PSL}(2,\bbZ)$ is naturally
identified with the set of all Euclidean NET maps with postcritical
set $P_2$ which fix $\zp_2(0)$.  On the other hand, a projective class
$[\mathbf{D}]$ of presentation diagrams determines a well-defined
isotopy class of Thurston maps $F: (S^2_2, P_2) \to (S^2_2, P_2)$ with
postcritical set $P_2$ (Lemma \ref{lemma:well-defined}).  We show:

\begin{thm} \label{thm:twisting1} Suppose $A\in \text{Mat}^+(2,\bbZ)$
sends $[\mathbf{D}]$ to $[\mathbf{D'}]$. Let $F, F': (S^2_2, P_2) \to
(S^2_2, P_2)$ be the corresponding Thurston maps. Then $F'=A_2\circ
F$, where $A_2: (S^2_2,P_2) \to (S^2_2,P_2)$ is the map induced by
$A$.
\end{thm} 

We prove Theorem~\ref{thm:twisting1} in Section~\ref{sec:twisting}.
We also obtain the corresponding result for translations in
Theorem~\ref{thm:twisting2}.  Together, they give a correspondence
between twists of projective classes of NET map presentation diagrams
and twists of NET maps by the full modular group of $F$ as well as by
all Euclidean NET maps with the same postcritical set as $F$.
\medskip

{\bf Outline.} \S \ref{sec:corre} collects some notation and technical
results. \S \ref{sec:diagram2map} defines NET map presentation
diagrams precisely and shows how a NET map presentation diagram
determines a NET map.  \S \ref{sec:twisting} proves
Theorem \ref{thm:twisting1} and the supplementary
Theorem~\ref{thm:twisting2}, which deals with translations.

The remainder of the paper is devoted to showing that general NET maps
admit NET map presentations.  The basic idea is to show first that an
arbitrary NET map $f$ has the form $f=h'\circ g$ where $g$ is a
Euclidean NET map as above and where now $h'$ is a push-point
homeomorphism along a set of possibly very complicated pairwise disjoint
arcs.  The proof of Theorem \ref{thm:pren} then exploits the fact that
$h'$ can be simplified at the expense of modifying the affine map $g$.  

\S \ref{sec:euc} discusses Euclidean NET maps in detail and shows that
they arise as compositional factors of general NET maps.  \S
\ref{sec:prelim} and \S \ref{sec:relate} attach preliminary data to NET
maps and show how this changes under twisting. \S \ref{sec:prens}
proves our main result. The proof we give of Theorem \ref{thm:pren},
while constructive in principle, is nevertheless somewhat implicit. In
Section~\ref{sec:algo} we give an algorithm for finding NET map
presentations, and illustrate it for two examples: the Douady rabbit
quadratic polynomial, and the cubic rational map studied by Lodge
\cite{l}.

Following this paper, \cite{fpp2} investigates modular groups, Hurwitz
classes and dynamic portraits of NET maps.  Next, \cite{fkklpps} is
partly a survey of what we know, but mostly a compilation of special
NET map results, both theoretical and computational.  The NET map web
site \cite{NET} contains papers, a census of many examples, and
executable files for the computer program {\tt NETmap} that generated
the census data.

\subsection*{Acknowledgements}
The authors gratefully acknowledge support from the American Institute
for Mathematics.  Kevin Pilgrim was also supported by Simons grant \#245269.

\section{Lattices, tori, spheres, and affine maps}
\label{sec:corre}\nosubsections

A lattice $\Lambda<\mathbb{R}^2$ determines a group of isometries
$\Gamma_\Lambda:=\{x \mapsto 2\lambda \pm x : \lambda \in \Lambda\}$
generated by 180-degree rotations about elements of $\Lambda$ and a
quotient map $\pi_\Lambda\co
\mathbb{R}^2\to\mathbb{R}^2/\Gamma_\Lambda$.  Given an ordered basis
$(\lambda, \mu)$ for $\Lambda$, the parallelogram in $\mathbb{R}^2$
spanned by $2\lambda$ and $2\mu$ is a fundamental domain for the
quotient torus $\mathbb{R}^2/2\Lambda$; the parallelogram spanned by
$2\lambda$ and  $\mu$ is a fundamental domain $F_\Lambda$ for the quotient
sphere $\mathbb{R}^2/\Gamma_\zL$.

The following two lemmas will be used repeatedly in the sequel; the
proofs are straightforward and are omitted.

\begin{lemma}\label{lemma:inducedown} Let $\zL$ and $\zL'$ be lattices
in $\mathbb{R}^2$.  Let $\zG$, respectively $\zG'$, be the groups of
Euclidean isometries of the form $x\mapsto 2\zl\pm x$ for some $\zl\in
\zL$, respectively $\zl\in \zL'$.  Also let $\zp\co \mathbb{R}^2\to
\mathbb{R}^2/\zG$ and $\zp'\co \mathbb{R}^2\to \mathbb{R}^2/\zG'$ be
the canonical quotient maps.  Let $\zF\co \mathbb{R}^2\to
\mathbb{R}^2$ be an affine isomorphism such that $\zF(\zL)\subseteq
\zL'$.  Then $\zF$ induces a branched covering map $\zf\co
\mathbb{R}^2/\zG\to \mathbb{R}^2/\zG'$ such that $\zf\circ
\zp=\zp'\circ \zF$ as in Figure~\ref{fig:induce}.  The map $\zF$
preserves orientation if and only if $\zf$ preserves orientation.  The
set $\zF(\zL)$ is a coset of a sublattice $\zL''$ of $\zL'$, and the
degree of $\zf$ equals the index $[\zL':\zL'']$.
\end{lemma}

\begin{figure}
  \begin{equation*}
\xymatrix{\mathbb{R}^2\ar[d]^{\zp}\ar[r]^\zF &
{\mathbb{R}^2}\ar[d]^{\zp'}\\
\mathbb{R}^2/\zG \ar[r]^\zf & \mathbb{R}^2/\zG'}
  \end{equation*}
 \caption{The maps of Lemmas~\ref{lemma:inducedown} and
\ref{lemma:induceup}}
\label{fig:induce}
\end{figure}

\begin{lemma}\label{lemma:induceup} Let $\zL$ and $\zL'$ be lattices
in $\mathbb{R}^2$.  Let $\zG$, respectively $\zG'$, be the group of
Euclidean isometries of the form $x\mapsto 2\zl\pm x$ for some $\zl\in
\zL$, respectively $\zL'$.  Also let $\zp\co \mathbb{R}^2\to
\mathbb{R}^2/\zG$ and $\zp'\co \mathbb{R}^2\to \mathbb{R}^2/\zG'$ be
the canonical quotient maps.  Let $\zf\co \mathbb{R}^2/\zG\to
\mathbb{R}^2/\zG'$ be a branched covering map such that
$\zf(\zp(\zL))\subseteq \zp'(\zL')$.
Then we have the following three statements.
\begin{enumerate}
  \item There exists a homeomorphism $\zF\co \mathbb{R}^2\to
\mathbb{R}^2$ such that the restriction of $\zF$ to $\zL$ is affine
and $\zf\circ \zp=\zp'\circ\zF$.  If $\zp'(0)\in \zf(\zp(\zL))$, then
$\zF(\zL)$ is a sublattice of $\zL'$.
  \item There exists an affine isomorphism $\zJ\co \mathbb{R}^2\to
\mathbb{R}^2$ such that the branched map of
Lemma~\ref{lemma:inducedown} which $\zJ$ induces from
$\mathbb{R}^2/\zG$ to $\mathbb{R}^2/\zG'$ is $\zf$ up to isotopy rel
$\zp(\zL)$.  If $\zp'(0)\in \zf(\zp(\zL))$, then $\zJ(\zL)$ is a
sublattice of $\zL'$.
  \item The maps $\zF$ and $\zJ$ are unique up to precomposing with an
element of $\zG$.  They are also unique up to postcomposing with an
element of $\zG'$.
\end{enumerate}
\end{lemma}

\section{From presentations and diagrams to maps}
\label{sec:diagram2map}

In all that follows, $\Lambda_2$ will denote the standard integer
lattice $\mathbb{Z}^2 < \mathbb{R}^2$, and $\Phi: \mathbb{R}^2 \to
\mathbb{R}^2$ will denote an affine map $x \mapsto Ax+b$ with the
property that $\Lambda_1:=\Phi(\Lambda_2)$ is a proper sublattice of
$\Lambda_2$. For $i=1,2$ we denote by $\pi_i$ the corresponding
quotient map as in \S \ref{sec:corre} and by $S^2_i$ the corresponding
quotient sphere; we put $P_i:=\pi_i(\Lambda_i)$.  The sphere $S_i^2$
has an affine structure, making it a half-translation sphere with set
of corner points $P_i$.

Precisely described, the data of a NET map presentation with
associated diagram $\bD$ consists of three things. Here are the first
two:
\begin{enumerate} 
  \item An ordered basis $(\zl_1, \zl_2)$ for $\Lambda_1$. This is
indicated in $\bD$ by drawing a parallelogram $F_1$ with sides joining
the origin to lattice points $2\zl_1, \zl_2$.  The choice between the
two possible orderings of the basis elements is uniquely determined by
the condition that the determinant of the column matrix $A=[\zl_1,
\zl_2]$ is positive. The intersection $\{0, \zl_1, 2\zl_1, \zl_2,
\zl_1+\zl_2, 2\zl_1+\zl_2\}$ of $\Lambda_1$ with $F_1$ is indicated in
$\bD$ by large solid dots; note that $P_1=\pi_1(\Lambda_1\cap F_1)$.
  \item An element $b \in \{0, \zl_1, \zl_2, \zl_1+\zl_2\}$.  It is
indicated in $\bD$ by circling the corresponding point in the boundary
of the parallelogram $F_1$.  
\end{enumerate}

The data of (1) and (2) above determines uniquely an
orientation-preserving affine isomorphism $\Phi: \mathbb{R}^2 \to
\mathbb{R}^2$ sending $\Lambda_2$ isomorphically to $\Lambda_1$ and
sending the origin to $b$; thus it is linear iff $b=0$.

From \S \ref{sec:corre}, we have the following.  Here,
$\mathbb{R}^2_1=\mathbb{R}^2_2=\mathbb{R}^2$; the subscripts are for
clarity.  We set $S^2_i=\mathbb{R}_i^2/\Gamma_i$ for $i=1,2$.  The
identity map $\mathrm{id}: \mathbb{R}^2_1 \to \mathbb{R}^2_2$ induces
a degree $d:=[\Lambda_2:\Lambda_1]$ orientation-preserving branched
covering $\overline{\mathrm{id}}: (S^2_1, P_1)\to (S^2_2,P_2)$.  (Not
only do we have that $\overline{\text{id}}(P_1)\subseteq P_2$ but even
that $\overline{\text{id}}(\zp_1(\zL_2))=P_2$.)  The affine
isomorphism $\Phi: \mathbb{R}_2^2 \to \mathbb{R}_1^2$ induces an
isomorphism $\phi: (S^2_2, P_2) \to (S^2_1, P_1)$. The composition
$\phi \circ \overline{\text{id}}: (S^2_1,P_1) \to (S^2_1, P_1)$, is a
Euclidean Thurston map, $g$. See Figure \ref{fig:defg}.  The lower
vertical arrows are the identity map.

\begin{figure}
  \begin{equation*}
\xymatrix{ 
\mathbb{R}_1^2\ar[d]^{\pi_1}\ar[r]^{\text{id}} &
\mathbb{R}_2^2\ar[d]^{\pi_2}\ar[r]^\zF & 
\mathbb{R}_1^2\ar[d]^{\pi_1} \\
\mathbb{R}_1^2/\zG_1\ar[d]^{id}\ar[r]^{\overline{\text{id}}} & 
\mathbb{R}_2^2/\zG_2\ar[r]^\zf & 
\mathbb{R}_1^2/\zG_1\ar[d]^{id} \\
S^2_1\ar[rr]^g & & S^2_1}
  \end{equation*}
 \caption{ Defining $g$}
\label{fig:defg}
\end{figure}

We call the data in (1) and (2) above {\em Euclidean NET map data}.
In the following \S \ref{sec:euc}, we show that up to equivalence, every
Euclidean NET map arises in this way.

The third and final ingredient in a NET map presentation is the
following. Recall that the four corners of the half-translation sphere
form the set $P_1=\pi_1(\Lambda_1)$, and that $\Lambda_1 \cap F_1 = \{0,
\zl_1, 2\zl_1, \zl_2, \zl_1+\zl_2, 2\zl_1+\zl_2\}$.  For convenience
of exposition, we denote $P_1=\{p_1, p_2, p_3, p_4\}$; the choice of
such a labeling is not part of the data of a NET map presentation.

\begin{enumerate} \setcounter{enumi}{2} \item For each $p_j \in P_1$,
a ``green'' (local) geodesic arc $\zb_j$ joining $p_j$ to an element of
$\pi_1(\Lambda_2)$, such that the geodesics $\zb_j$ are pairwise
disjoint.  We allow the case of degenerate (zero length) geodesics.
We require that every $\zb_j$ lifts to a line segment in $F_1$.  The
choice of the $\zb_j$'s is represented in $\bD$ by drawing all of
their lifts in the closed fundamental domain $F_1$ as green line
segments.  Thus there may be up to eight such segments.
\end{enumerate}
Let $P$ be the set of four points in $S_1^2$ consisting of the
terminal endpoints of the $\zb_j$'s.  The data in (3) determines a
{\em push-point homeomorphism} $h: (S_1^2, P_1) \to (S_1^2, P)$,
unique up to isotopy rel $P_1$, by pushing the four points of $P_1$
along the (possibly trivial) oriented green geodesics $\zb_j$.

The result below follows immediately from straightforward arguments.

\begin{lemma} \label{lemma:well-defined} The map $f=h\circ g=h\circ
\zf\circ \overline{\text{id}}\co (S_1^2,P)\to (S_1^2,P)$ is a Thurston
map; it is a NET map if $\#P_f=4$.  Once the data (1), (2), and (3)
are fixed, the following are uniquely determined:
\begin{enumerate} 
  \item the subset $P_f\subset P\subset S^2_1$; 
  \item the isotopy class of $f$ relative to $P\subset S^2_1$; 
  \item the isotopy class of the conjugate
$F:=\overline{\mathrm{id}}\circ f\circ \overline{\mathrm{id}}^{-1} =
\overline{\mathrm{id}}\circ h \circ \phi: (S_2^2,P_2) \to (S_2^2,
P_2)$; $P_F \subset P_2$ and equality holds iff $F$ (equivalently,
$f$) is a NET map; 
  \item the equivalence class of $f$.
\end{enumerate}
\end{lemma}
If $\#P_f=4$ (equivalently, if $P_f=P$) then we call the data in
(1), (2), (3) above {\em NET map data} and we say that (1), (2), (3)
give a NET map presentation for $f$.  This presentation  is encoded in
the diagram $\bD$.
\medskip

\noindent{\bf Remarks:}
\begin{enumerate} \item The dynamical plane---the sphere $S^2_1$---and
postcritical set of $f$ depend on the data in (1), (2), (3), while
those of its conjugate $F$ do not. The former representation is more
convenient for computations of dynamical invariants like slope
functions, while the latter is more convenient in other situations
since the set of isotopy classes of Thurston maps with a common
postcritical set $P_2$ is naturally a biset over the pure mapping
class group $P\Gamma(2)$.

\item Suppose a diagram $\mathbf{D}$ is fixed; let
$P\Gamma(2).[\mathbf{D}]$ be the orbit of its projective
class under the natural free action of the pure mapping class group
$P\Gamma(2)$.  Theorem \ref{thm:twisting1} implies that the image of
the orbit is a subset of isotopy classes of NET maps that naturally
forms an irreducible biset $\mathcal{F}$ over the group $P\Gamma(2)$
in the sense of \cite{N}.  If some (equivalently, any) map in this
orbit is not Euclidean, then the natural map $[\mathbf{D}]
\mapsto F$ is never injective on the level of isotopy classes, since
by \cite[Theorem 6.3]{kps} there will always be elements $A$ for which
$A_2\circ F=F$ up to isotopy.
\end{enumerate}

\section{Affine twisting of diagrams}
\label{sec:twisting}

In this section, we prove Theorems~\ref{thm:twisting1} and
\ref{thm:twisting2}.

Let $\mathbf{D}$ be a NET map presentation diagram, $f: (S^2_1, P) \to
(S^2_1, P)$ the corresponding NET map with postcritical set $P$, and
$F: (S^2_2, P_2) \to (S^2_2, P_2)$ the corresponding conjugate. Recall
from Lemma \ref{lemma:well-defined} that the isotopy class of $F$
relative to $P_2$ is independent of the choice used in the definition
of push homeomorphism.  Let $A \in \text{Mat}^+(2,\bbZ)$, let
$\mathbf{D}':=A\cdot\mathbf{D}$ be the NET map presentation diagram
obtained by applying the linear map induced by $A$ to the diagram
$\mathbf{D}$, and let $F': (S^2_2, P_2) \to (S^2_2, P_2)$ be the
corresponding NET map.

\begin{proof}[Proof of Theorem~\ref{thm:twisting1}] For clarity of
exposition, we use primes to distinguish objects associated with
$\mathbf{D}'$, and, as usual, subscripts $1$ and $2$ to distinguish
between quotients associated to lattices $\Lambda_1, \Lambda_2$. 

We have $\Lambda_\onep=A\cdot\Lambda_1$ by the definition of the
action, so the diagram in Figure~\ref{fig:ids} commutes.  Hence,
abusing notation, $\Phi' = A\circ \Phi$ and $\phi' = A_1\circ \phi$.

\begin{figure}
\centerline{\xymatrix{
(\mathbb{R}^2, \Lambda_1)\ar[r]^{A\cdot}\ar[d]_{\pi_1}& 
(\mathbb{R}^2, \Lambda_{\onep})\ar[d]^{\pi_\onep}\\
(S^2_1, P_1)\ar[r]^{A_1}\ar[d]_{\overline{\mathrm{id}}}& 
(S^2_\onep, P_\onep)\ar[d]^{\overline{\mathrm{id}'}}\\
(S^2_2, P_2) \ar[r]^{A_2}&(S^2_2, P_2) \\
}}
\caption{Maps induced by the matrix $A$}
\label{fig:ids}
\end{figure}

We next consider the effect on the push homeomorphism. Again by the
definition of the action, $A_1$ sends the green arcs $\zb_j$ to the
green arcs $\zb'_j$. It follows that the push homeomorphisms $h_1$
and $h_\onep$ (subscript 1's added for clarity) in the definitions of
$f$ and of $f'$ may be chosen so that 
\begin{equation}
\label{eqn:h}
h_\onep  = A_1\circ  h_1\circ  A_1^{-1}.
\end{equation}

Now consider the diagram in Figure~\ref{fig:FF'}.  The triangle commutes
because $\phi' = A_1\circ \phi$.  Line~\ref{eqn:h} shows that the left
square commutes.  The bottom square of Figure~\ref{fig:ids} shows that
the right square commutes.  Thus the diagram in Figure~\ref{fig:FF'}
commutes, and so $F'=A_2\circ F$.  

This provesTheorem~\ref{thm:twisting1}.

\begin{figure}
\centerline{\xymatrix@R=.5pc{ 
& S_1^2\ar[dd]^{A_1}\ar[r]^{h_1} &
S_1^2\ar[dd]^{A_1}\ar[r]^{\overline{\text{id}}} & 
S_2^2\ar[dd]^{A_2} \\
S_2^2\ar[ur]_{\zf}\ar[dr]^{\zf'}\ar@/^3pc/[urrr]^(.6)F
\ar@/_3pc/[drrr]_(.6){F'} \\
& S_{1'}^2\ar[r]^{h_{1'}} &
S_{1'}^2\ar[r]^{\overline{\text{id}'}} & S_2^2}}
\caption{Proving Theorem~\ref{thm:twisting1}}
\label{fig:FF'}
\end{figure}

\end{proof}

The rest of this section is devoted to the analog of
Theorem~\ref{thm:twisting1} for translations.  We maintain the
notation $\zL_2$, $\zG_2$, $\zp_2$, $S_2^2$ and $P_2$ of
Theorem~\ref{thm:twisting1}.  Now let $b$ be a vector in
$\zL_2=\bbZ^2$.  The translation $x\mapsto x+b$ on $\bbR^2$ induces a
homeomorphism $b_2\co (S_2^2,P_2)\to (S_2^2,P_2)$, an element of the
modular group of $(S_2^2,P_2)$.  Because $\zG_2$ contains all
translations by elements in $2\zL_2$, this homeomorphism depends only
on the coset $b+2\zL_2$ in $\zL_2/2\zL_2$.  We identify $\zL_2/2\zL_2$
with $\bbZ_2^2$ by means of the canonical group homomorphism from
$\bbZ^2$ to $\bbZ_2^2$.  In this way $b$ determines an element
$\overline{b}\in \bbZ_2^2$.
We obtain another element of $\bbZ_2^2$ from a NET map presentation
diagram $\bD$ as follows.  Let $\zL_1$ be the lattice associated to
$\bD$ with ordered basis $(\zl_1,\zl_2)$.  One of the lattice points
$\{0, \zl_1, \zl_2, \zl_1+\zl_2\}$ in the boundary of $\bD$ is
circled.  Call it $a$.  There is a group homomorphism $\zL_1\to
\bbZ_2^2$ which sends $\zl_1$ to $(1,0)$ and $\zl_2$ to $(0,1)$. Let
$\overline{a}$ denote the image in $\bbZ_2^2$ of $a$ under this group
homomorphism.  Now we have an element $\overline{b}$ of $\bbZ_2^2$
associated to $b$ and an element $\overline{a}$ of $\bbZ_2^2$
associated to $\bD$.  We obtain a new NET map presentation diagram
$\bD'$ such that $\bD'$ is the same as $\bD$ except that now the
translation term corresponding to $\bD'$ is
$\overline{b}+\overline{a}$.
The map $\bD\mapsto
\bD'$ extends to projective classes $[\bD]\mapsto [\bD']$.  We have:

\begin{thm}\label{thm:twisting2} Suppose $b\in \bbZ^2$ sends $[\bD]$
to $[\bD']$.  Let $F,F'\co (S_2^2,P_2)\to (S_2^2,P_2)$ be the
corresponding Thurston maps.  Then $F'=b_2\circ F$, where $b_2\co
(S_2^2,P_2)\to (S_2^2,P_2)$ is the map induced by $b$.
\end{thm}
 \begin{proof} Let $e_1=(1,0)$ and $e_2=(0,1)$, the standard basis
vectors in $\bbZ^2$.  We have a straightforward bijection $\zt\co
\{0,e_1,e_2,e_1+e_2\}\to \{0,\zl_1,\zl_2,\zl_1+\zl_2\}$.  The
homeomorphism $\zf$ conjugates the element in the modular group of
$(S_2^2,P_2)$ induced by $x\mapsto x+e$, where $e\in
\{0,e_1,e_2,e_1+e_2\}$, to the element of the modular group of
$(S_1^2,P_1)$ induced by $x\mapsto x+\zt(e)$.
Theorem~\ref{thm:twisting2} follows easily from this.

\end{proof}

\section{Presentation and equivalence of Euclidean NET maps }
\label{sec:euc}\nosubsections

\medskip\noindent \textbf{Completeness of construction.}  In this
paragraph, we show that the conditions in data items (1) and (2) of \S
\ref{sec:diagram2map} obtains every Thurston equivalence class of
Euclidean NET maps.  So, let $g\co S^2\to S^2$ be a Euclidean NET map.
As in Section 1 of \cite{cfpp}, we construct lattices $\zL_i$, groups
$\zG_i$ of Euclidean isometries and quotient maps $\zp_i\co
\mathbb{R}^2\to \mathbb{R}^2/\zG_i=S^2_i$ for $i\in \{1,2\}$.  There
exist identifications of $S^2_i$ with the dynamical plane $S^2$ of $g$
so that $\zp_i(\zL_i)$ is the postcritical set $P_g$ of $g$ for $i\in
\{1,2\}$ and $g\circ \zp_1=\zp_2$.  We have the right half of
Figure~\ref{fig:complete}.  Now we apply statement 2 of
Lemma~\ref{lemma:induceup} with $\zL=\zL_2$, $\zL'=\zL_1$ and $\zf\co
S^2\to S^2$ the identity map.  We obtain an affine isomorphism $\zF\co
\mathbb{R}^2\to \mathbb{R}^2$ such that $\zF(\zL_2)$ is a sublattice
of $\zL_1$ and the branched map which $\zF$ induces from $S^2$ to
$S^2$ is isotopic to the identity map rel $P_g$.  Because $\zF$
induces the identity map, $\zF(\zL_2)=\zL_1$.  Because $g$ has degree
greater than 1, the lattice $\zL_1$ is a proper sublattice of $\zL_2$.
This proves that the construction is complete.

\begin{figure}
  \begin{equation*}
\xymatrix{\mathbb{R}^2\ar[d]^{\zp_2}\ar[r]^\zF & 
\mathbb{R}^2\ar[d]^{\zp_1}\ar[r]^{\text{id}} &
\mathbb{R}^2\ar[d]^{\zp_2}\\
S^2\ar[r]^{id} & S^2\ar[r]^g & S^2\\}
  \end{equation*}
 \caption{Proving that the construction is complete}
\label{fig:complete}
\end{figure}

\medskip\noindent \textbf{Equivalence of Euclidean NET map data.}
Thus far we have a way to construct Euclidean NET maps from our
Euclidean NET map data, and we know that every Euclidean NET map has
this form.  Hence the Thurston equivalence relation on Euclidean NET
maps corresponds to an equivalence relation on sets of Euclidean NET
map data.  Here it is.

\smallskip\noindent\textbf{Equivalence relation on data.}  Suppose
that we have lattices $\zL_2\subseteq \mathbb{R}^2$ and
$\zL'_2\subseteq \mathbb{R}^2$ and orientation-preserving affine
isomorphisms $\zF\co \mathbb{R}^2\to \mathbb{R}^2$ and $\zF'\co
\mathbb{R}^2\to \mathbb{R}^2$ such that $\zF(\zL_2)$ is a proper
sublattice $\zL_1$ of $\zL_2$ and $\zF'(\zL'_2)$ is a proper
sublattice $\zL'_1$ of $\zL'_2$.  We say that the pair $(\zL_2,\zF)$
is equivalent to the pair $(\zL'_2,\zF')$ if and only if there exists
an orientation-preserving affine isomorphism $\zJ\co \mathbb{R}^2\to
\mathbb{R}^2$ such that
\begin{enumerate}
  \item $\zJ(\zL_2)=\zL'_2$;
  \item $\zJ\circ \zF\circ\zJ^{-1}=\pm\zF'+2\zl$ for some $\zl\in\zL_1'$.
\end{enumerate}

This defines an equivalence relation on data pairs.  Using
Lemmas~\ref{lemma:inducedown} and \ref{lemma:induceup}, it is a
straightforward exercise to prove that data equivalence implies map
equivalence and map equivalence implies data equivalence; see Figure \ref{fig:eqverln}.

\begin{figure}
  \begin{equation*}
\xymatrix{ & \mathbb{R}^2\ar'[d][dd]^>>>>>>
{\zp'_1}\ar[rr]^{\zF'} & & \mathbb{R}^2\ar[dd]^{\zp'_1}\\
\mathbb{R}^2\ar[dd]_{\zp_1}
\ar[rr]^>>>>>>>>>>{\qquad\zF}\ar[ur]^\zJ &
 & \mathbb{R}^2\ar[dd]_>>>>>>{\zp_1}\ar[ur]^\zJ & \\
 & S^2\ar'[r][rr]^<<<<{g'} & & S^2\\
S^2\ar[rr]^g\ar[ur]^\zj & & S^2\ar[ur]^\zj &}
  \end{equation*}
 \caption{Equivalence of Euclidean NET maps}
\label{fig:eqverln}
\end{figure}

\section{Preliminary presentations of general NET maps }
\label{sec:prelim}\nosubsections

In this section we extend the results of the previous section to
general NET maps by showing that each NET map $f$ is equivalent to a
map presented by what we call preliminary data. This data will factor
$f$ as a composition of a Euclidean NET map $g$ and a
push-homeomorphism $h$ where the arcs $\beta_j$ along which the pushes
occur are pairwise disjoint.

Here is our preliminary presentation for NET maps.

\medskip\noindent
\textbf{Preliminary NET map data.}
\begin{enumerate}
  \item A lattice $\zL_2\subseteq \mathbb{R}^2$.
  \item An orientation-preserving affine isomorphism
$\zF\co\mathbb{R}^2\to \mathbb{R}^2$ such that $\zF(\zL_2)$ is a
proper sublattice $\zL_1$ of $\zL_2$.
  \item Four closed arcs $\za_1$, $\za_2$, $\za_3$,
$\za_4$ in $\mathbb{R}^2$ such that $\za_i$ has initial endpoint in
$\zL_1$ and terminal endpoint in $\zL_2$ for every $i\in \{1,2,3,4\}$.
Furthermore, the images of $\za_1$, $\za_2$, $\za_3$, $\za_4$ in
$\mathbb{R}^2/\zG_1$ are disjoint, where $\zG_1$ is the group of
Euclidean isometries of the form $x\mapsto 2\zl\pm x$ for some $\zl\in
\zL_1$.
\end{enumerate}

In the present case of general NET maps, our preliminary data do not
always determine an equivalence class of NET maps, but only because
the Thurston maps which we obtain might have fewer than four
postcritical points.

\medskip\noindent\textbf{Construction.}  We use the preliminary NET 
map data to construct an equivalence class of Thurston maps.  These
maps will be NET maps if they have at least four postcritical points,
which is almost always the case.  The data in items (1) and (2)
constitute Euclidean NET map data.  Let $g\co S^2\to S^2$ be a
Euclidean NET map associated to the data in items (1) and (2).

Now we turn our attention to the construction of a homeomorphism $h\co
S^2\to S^2$.  The construction of $g$ involves an identification of
$\mathbb{R}^2/\zG_1$ with $S^2$.  So $\za_1$, $\za_2$, $\za_3$,
$\za_4$ map to disjoint closed arcs $\zb_1$, $\zb_2$, $\zb_3$, $\zb_4$
in $S^2$.  We take $h\co S^2\to S^2$ to be a push map along
$\zb_1$, $\zb_2$, $\zb_3$, $\zb_4$.

Now that $g$ and $h$ are defined, we set $f=h\circ g$.  Let $P_1$ be
the set of initial endpoints of $\zb_1$, $\zb_2$, $\zb_3$, $\zb_4$,
and let $P$ be the set of terminal endpoints of $\zb_1$, $\zb_2$,
$\zb_3$, $\zb_4$.  The definitions imply that $P_1$ is the
postcritical set of $g$ and that $g(P)\subseteq P_1$.  It follows that
$f$ is a Thurston map, and it is a NET map if it has at least four
postcritical points.  In this case, $P$ is the postcritical set of
$f$.  This completes the construction of a Thurston map $f$ from the
preliminary NET map data.

We next show that the equivalence class of the Thurston map $f$ just
constructed is independent of the choices involved.  There are two
choices.  The first involves the identification of
$\mathbb{R}^2/\zG_1$ with $S^2$.  Changing this identification
corresponds to conjugating $f$ by a homeomorphism.  The result is
equivalent to $f$.  The second choice is the choice of $h$.  But $h$
is unique up to isotopy rel $P_1$.  Hence because $g(P)\subseteq
P_1$, changing $h$ amounts to changing $f$ by an isotopy rel $P$.
This is also equivalent to $f$.

\medskip\noindent\textbf{Completeness of construction.}  We show that
this construction obtains every equivalence class of NET maps.  So,
let $f\co S^2\to S^2$ be a NET map.  We have the usual lattices
$\zL_1$ and $\zL_2$, groups $\zG_1$ and $\zG_2$, maps $\zp_1$ and
$\zp_2$ and four-element sets $P_1$ and $P_2$ coming from Section 1
of \cite{cfpp}.  Lemma 1.3 of \cite{cfpp} shows that there exist
exactly four points in $S^2$ which are not critical points of $f$ such
that $f$ maps them to its postcritical set.  We choose an
identification of $S^2$ with $S_1^2$ so that $P_1$ is this set of four
points.  We choose an identification of $S^2$ with $S_2^2$ so that
$P_2$ is the postcritical set of $f$.  We construct four disjoint
closed arcs $\zb_1$, $\zb_2$, $\zb_3$, $\zb_4$ in $S^2$ such that the
initial endpoint of $\zb_i$ is in $P_1$ and the terminal endpoint of
$\zb_i$ is in $P_2$ for every $i\in \{1,2,3,4\}$.  We choose four
closed arcs $\za_1$, $\za_2$, $\za_3$, $\za_4$ in $\mathbb{R}^2$ such
that $\zp_1$ maps $\za_i$ homeomorphically to $\zb_i$ for every $i\in
\{1,2,3,4\}$.  Then the initial endpoint of $\za_i$ is in $\zL_1$ and
the terminal endpoint of $\za_i$ is in $\zL_2$ for every $i\in
\{1,2,3,4\}$.  Hence we have the NET map data in item (3).  Let $h\co
S^2\to S^2$ be a push map relative to $\zb_1$, $\zb_2$, $\zb_3$,
$\zb_4$.  We know that then $f=h\circ g$, where $g$ is a Euclidean NET
map.  Section~\ref{sec:euc} shows that $g$ is Thurston equivalent to a
Euclidean NET map $g'$ which arises from NET map data as in items (1)
and (2).  This means that there exist homeomorphisms $\zj_1\co S^2\to
S^2$ and $\zj_2\co S^2\to S^2$ which are isotopic rel $P_1$ such that
$g'\circ \zj_2=\zj_1\circ g$.  Let $h'=\zj_2\circ h\circ \zj_1^{-1}$,
and let $f'=h'\circ g'$.  We have the commutative diagram in
Figure~\ref{fig:cnstrncmp}.  We see that $f$ and $f'$ are conjugate,
hence Thurston equivalent.  Because $\zj_1$ and $\zj_2$ are isotopic
rel $P_1=P_g$, the maps $\zj_1^{-1}$ and $\zj_2^{-1}$ are isotopic rel
$P_{g'}$.  So $h'$ is isotopic rel $P_{g'}$ to $h''=\zj_2\circ h\circ
\zj_2^{-1}$, which is clearly a push map of the desired form.  So $f$
is Thurston equivalent to a NET map $f''=h''\circ g'$ which arises
from NET map data.  This shows that our construction obtains every
equivalence class of NET maps.

\begin{figure}
  \begin{equation*}
\xymatrix{S^2\ar[d]^{\zj_2}\ar[r]^g &
S^2\ar[d]^{\zj_1}\ar[r]^h & S^2\ar[d]^{\zj_2}\\
S^2\ar[r]^{g'} & S^2\ar[r]^{h'} & S^2}
  \end{equation*}
 \caption{Proving completeness of the construction}
\label{fig:cnstrncmp}
\end{figure}

\medskip\noindent\textbf{Conjugation equivalence of NET map data.}
Just as for Euclidean NET map data, the Thurston equivalence relation
on NET maps corresponds to an equivalence relation on preliminary NET
map data.  We are about to define a conjugation equivalence relation
on preliminary NET map data analogous to the equivalence relation on
Euclidean NET map data.  Unlike the Euclidean case, this conjugation
equivalence relation is in general strictly stronger than the
equivalence relation on preliminary NET map data which corresponds to
the Thurston equivalence relation on NET maps.

\medskip\noindent\textbf{Conjugation equivalence relation on 
preliminary data.}  Suppose that we have two sets of NET map data.  So
we have lattices $\zL_2\subseteq \mathbb{R}^2$ and $\zL'_2\subseteq
\mathbb{R}^2$ and orientation-preserving affine isomorphisms $\zF\co
\mathbb{R}^2\to \mathbb{R}^2$ and $\zF'\co \mathbb{R}^2\to
\mathbb{R}^2$ such that $\zF(\zL_2)$ is a proper sublattice $\zL_1$ of
$\zL_2$ and $\zF'(\zL'_2)$ is a proper sublattice $\zL'_1$ of
$\zL'_2$.  We also have closed arcs $\za_1$, $\za_2$, $\za_3$, $\za_4$
in $\mathbb{R}^2$ such that $\za_i$ has initial endpoint in $\zL_1$
and terminal endpoint in $\zL_2$ and closed arcs $\za'_1$, $\za'_2$,
$\za'_3$, $\za'_4$ in $\mathbb{R}^2$ such that $\za'_i$ has initial
endpoint in $\zL'_1$ and terminal endpoint in $\zL'_2$.  Furthermore,
the images $\zb_1$, $\zb_2$, $\zb_3$, $\zb_4$ of $\za_1$, $\za_2$,
$\za_3$, $\za_4$ in $\mathbb{R}^2/\zG_1$ are disjoint, and the images
$\zb'_1$, $\zb'_2$, $\zb'_3$, $\zb'_4$ of $\za'_1$, $\za'_2$,
$\za'_3$, $\za'_4$ in $\mathbb{R}^2/\zG'_1$ are disjoint.  We say that
the triple $(\zL_2,\zF,\{\za_1,\za_2,\za_3,\za_4\})$ is conjugation
equivalent to the triple
$(\zL'_2,\zF',\{\za'_1,\za'_2,\za'_3,\za'_4\})$ if and only if there
exists an orientation-preserving affine isomorphism $\zJ\co
\mathbb{R}^2\to \mathbb{R}^2$ such that
\begin{enumerate}
  \item $\zJ(\zL_2)=\zL'_2$;
  \item $\zJ\circ \zF\circ \zJ^{-1}=\pm \zF'+2\zl$ for some $\zl\in\zL'_1$;
  \item There is an isotopy of $\mathbb{R}^2/\zG'_1$  rel
        $P'_1\cup P'_2$ taking $\zj\circ\zb_1$, $\zj\circ \zb_2$,
        $\zj\circ \zb_3$, $\zj\circ \zb_4$ to $\zb'_1$, $\zb'_2$,
    $\zb'_3$, $\zb'_4$, not necessarily in order, where
        $\zj\co \mathbb{R}^2/\zG_1\to \mathbb{R}^2/\zG'_1$ is the map
    of Lemma~\ref{lemma:inducedown} induced by $\zJ$ .
\end{enumerate}

This defines an equivalence relation on preliminary data triples.  One
can show as in Section~\ref{sec:euc} that if two sets of preliminary
data are conjugation equivalent, then the corresponding NET maps are
equivalent.

\section{The effect of twisting on preliminary data}
\label{sec:relate}\nosubsections

Suppose $f$ is a NET map, and suppose that we have expressed $f$ in
terms of preliminary data as in the previous subsection, so that $f$
factors as $f=h\circ g$. The arcs $\beta_j$ along which the
push-homeomorphism $h$ occurs may be very wild.  In
Section~\ref{sec:prens} we will show that we can absorb their
complexity into the affine part $g$. To do this we now examine the
effects of twisting on preliminary data.  This result is closely
related to Theorems~\ref{thm:twisting1} and \ref{thm:twisting2}.

\begin{thm}\label{thm:leftaction} Let $f$ be a NET map.  We construct
lattices and maps $\zL_1$, $\zL_2$, $\zp_1$, $\zp_2$ as in the last
section's proof of the completeness of the construction.  Relative to
these maps $f$ decomposes as $f=h\circ g$, where $g$ is a Euclidean
NET map and $h$ is a homeomorphism.  Suppose that $g$ lifts via
$\zp_1$ to an affine isomorphism $\zF\co \mathbb{R}^2\to \mathbb{R}^2$
such that $\zF(\zL_2)=\zL_1$.  Let $\zj\co (S^2,P)\to (S^2,P)$ be an
orientation-preserving branched covering map which lifts via $\zp_2$
to an affine isomorphism $\zJ\co \mathbb{R}^2\to \mathbb{R}^2$ such
that $\zJ(\zL_2)\subseteq \zL_2$.  Then $\zj\circ f=h\circ g'$, where
$g'$ is the Euclidean NET map which lifts via $\zp_1$ to $\zF\circ
\zJ$.
\end{thm}
  \begin{proof} As is usual for NET maps, we have that $\zp_2 =f\circ
\zp_1$.  So since $\zj$ lifts to $\zJ$ via $\zp_2$, we have the left
half of the commutative diagram in Figure~\ref{fig:leftaction}.  The
rest of Figure~\ref{fig:leftaction} results from the facts that
$f=h\circ g$ and that $g$ lifts to $\zF$ via $\zp_1$.  It need not be
true that $\zJ(\zL_1)=\zL_1$, so there need not be an induced
homeomorphism from $S^2$ to $S^2$ subdividing the left half of
Figure~\ref{fig:leftaction} into two squares.  However, $(\zF\circ
\zJ)(\zL_1)\subseteq (\zF\circ \zJ)(\zL_2)\subseteq \zF(\zL_2)=\zL_1$,
and so Lemma~\ref{lemma:inducedown} with $\zL=\zL'=\zL_1$ implies that
$\zF\circ \zJ$ induces a branched covering map $g'\co S^2\to S^2$ via
$\zp_1$.  This and the surjectivity of $\zp_1$ imply that $\zj\circ
f=h\circ g'$.

\begin{figure}
  \begin{equation*}
\xymatrix{\mathbb{R}^2\ar[d]^{\zp_1}\ar[r]^\zJ &
\mathbb{R}^2\ar[d]^{\zp_1}\ar[r]^\zF &
\mathbb{R}^2\ar[d]^{\zp_1}\\
S^2\ar[d]^f & S^2\ar[d]^f\ar[r]^g & S^2\ar[d]^h\\
S^2\ar[r]^\zj & S^2\ar[r]^{id} & S^2}
  \end{equation*}
 \caption{ Proving Theorem~\ref{thm:leftaction}}
\label{fig:leftaction}
\end{figure}

This proves Theorem~\ref{thm:leftaction}.

\end{proof}

\noindent\textbf{Remark.} One might find it unsettling that the left
action of $\zj$ on $f$ in Theorem~\ref{thm:leftaction} seems to
correspond to a right action of $\zJ$ on $\zF$.  However, keep in mind
that $\zF$ lifts $g$ via $\zp_1$ and that $\zJ$ lifts $\zj$ via
$\zp_2$.  While $\zp_1$ does not change in passing from $f$ to
$\zj\circ f$, the map $\zp_2$ does change.  Hence there is no right
action here.  \bigskip

\section{NET map presentations }\label{sec:prens}\nosubsections

In this section we prove Theorem \ref{thm:pren}.

Recall that in Section~\ref{sec:diagram2map} we defined a NET map
presentation to consist of three things.  The following optional
condition can also be imposed if desired.

\begin{enumerate}\setcounter{enumi}{3}
  \item There exist integers $m\ge 2$ and $n\ge 1$ such
that i) $n$ divides $m$, ii) $m$ divides $\zl_1$, iii) $n$ divides
$\zl_2$ and iv) $\text{det}(A)=mn$.
\end{enumerate}

The integers $m$ and $n$ are the \textbf{elementary divisors} of the
integer matrix $A:=[\lambda_1, \lambda_2]$.  Corollary 5.2 of
\cite{fpp2} shows that they are uniquely determined not only by the
Thurston equivalence class of a NET map, but even by its modular group
Hurwitz class.  We use the convention $n|m$ instead of the more common
condition $m|n$ so that pictures of associated fundamental domains
tend to be wider than they are tall and so are more convenient to
draw.

Having now an ordered basis $(\lambda_1, \lambda_2)$, we define an
associated fundamental domain for $\zG_1$ as the parallelogram with
corners 0, $2 \zl_1$, $\zl_2$ and $2 \zl_1+\zl_2$.

We are now prepared to prove Theorem~\ref{thm:pren}.

\begin{proof}[Proof of Theorem~\ref{thm:pren}] As a slight
strengthening of Theorem~\ref{thm:pren}, we prove that data item (4) can
be satisfied in addition to (1), (2) and (3).  The strategy of the proof is
to define first approximations of the objects in items (1) through (4).
These first approximations will be denoted by letters with prime
superscripts.  The final desired objects will eventually be gotten
from these by applying a linear transformation.

Let $f$ be a NET map.  Section~\ref{sec:prelim} shows that $f$ is
Thurston equivalent to a NET map which arises from the preliminary NET
map data presented there.  It is always possible to take the lattice
$\zL_2$ in preliminary NET map data item (1) to be $\mathbb{Z}^2$.
After making this normalization, let $\zF(x)=A'x+b'$ be the affine
isomorphism in preliminary NET map data item (2).  Since $\text{det}(A')
= \deg(f)\ge 2$, there exist integers $m\ge 2$ and $n\ge 1$ with $n|m$
and elements $P,Q\in \text{SL}(2,\mathbb{Z})$ such that
$A'=P\left[\begin{smallmatrix}m & 0 \\ 0 &
n\end{smallmatrix}\right]Q$.  Now we use the fact from
Section~\ref{sec:prelim} that the Thurston equivalence class of $f$ is
invariant under conjugation to replace $A'$ by $QA'Q^{-1}$.  As a
result, we may assume that $Q=1$.  So the first column $\zl'_1$ of
$A'$ is divisible by $m$, and the second column $\zl'_2$ of $A'$ is
divisible by $n$.  We have the data in NET map data item (4).  Let
$\zL'_1$ be the sublattice of $\zL_2$ generated by $\zl'_1$ and
$\zl'_2$.  The lattice $\zL'_1$ will lead to the desired lattice
$\zL_1$, and the vectors $\zl'_1$ and $\zl'_2$ will lead to the
vectors $\zl_1$ and $\zl_2$ in NET map data item (3).

We next define four line segments $\za'_1$, $\za'_2$, $\za'_3$,
$\za'_4$ which will lead to the line segments in NET
map data item (3).  Let $F'_1$ be the parallelogram in $\mathbb{R}^2$ with
corners 0, $2\zl'_1$, $\zl'_2$ and $2\zl'_1+\zl'_2$.  Then $F'_1$ is a
fundamental domain for the usual group $\zG'_1$.  The interior of $F'_1$
maps homeomorphically into $\mathbb{R}^2/\zG'_1$.  We refer to the
line segment with endpoints 0 and $2\zl'_1$ as the bottom of $F'_1$.  We
refer to the line segment with endpoints $\zl'_2$ and $2\zl'_1+\zl'_2$
as the top of $F'_1$.  Similarly, the line segment with endpoints 0 and
$\zl'_2$, respectively $2\zl'_1$ and $2\zl'_1+\zl'_2$, is the left,
respectively right, side of $F'_1$.  The rotation $x\mapsto 2\zl'_1-x$
identifies the two halves of the bottom of $F'_1$, and the rotation
$x\mapsto 2\zl'_1+2\zl'_2-x$ identifies the two halves of the top of
$F'_1$.  The translation $x\mapsto 2\zl'_1+x$ identifies the two sides
of $F'_1$.  So the union of the interior of $F'_1$ and the ``left half''
of the boundary of $F'_1$ maps bijectively to $\mathbb{R}^2/\zG'_1$.
Let $P_2$ be the postcritical set of $f$, and let $\widetilde{P}_2$ be
the inverse image of $P_2$ in this union.  So $\widetilde{P}_2$
consists of four points which are either in the interior of $F'_1$ or
in the left half of its boundary.

We begin to choose four line segments $\za'_1$, $\za'_2$, $\za'_3$,
$\za'_4$.  If $\widetilde{P}_2$ contains an element in the bottom of
$F'_1$, then we choose the element of $\widetilde{P}_2$ in the bottom of
$F'_1$ which is nearest 0, and we take one of our line segments to be
the line segment joining 0 and this point.  If $\widetilde{P}_2$
contains another element in the bottom of $F'_1$, then we choose the one
nearest $\zl'_1$, and we take one of our line segments to be the line
segment joining $\zl'_1$ and this point.  If $\widetilde{P}_2$
contains an element in the top of $F'_1$, then we choose line segments
analogously using $\zl'_2$ and $\zl'_1+\zl'_2$ instead of 0 and
$\zl'_1$.

Suppose that 0 is not yet in a line segment.  Then we choose an
element of $\widetilde{P}_2$ not yet chosen with minimal
$\zl'_2$-coordinate relative to the basis of $\mathbb{R}^2$ consisting
of $\zl'_1$ and $\zl'_2$.  We take the line segment joining this point
and 0.  Similarly, if $\zl'_1$ is not yet in a line segment, then we
choose an element of $\widetilde{P}_2$ not yet chosen with minimal
$\zl'_2$-coordinate, and we take the line segment joining this point
and $\zl'_1$.  The two line segments chosen thus far in this paragraph
might meet.  If so, then these two line segments are the diagonals of
a (possibly degenerate) quadrilateral, and we simply exchange them for
two opposite sides of that quadrilateral.  We choose line segments for
$\zl'_2$ and $\zl'_1+\zl'_2$ in the same way as for 0 and $\zl'_1$,
using maximal $\zl'_2$-coordinates instead of minimal
$\zl'_2$-coordinates.

We identify $\mathbb{R}^2/\zG'_1$ and $\mathbb{R}^2/\zG_2$ with $S^2$
as usual.  Although the instructions for choosing $\za'_1$, $\za'_2$,
$\za'_3$, $\za'_4$ are somewhat verbose, it is easy to see that the
images $\zb'_1$, $\zb'_2$, $\zb'_3$, $\zb'_4$ in $S^2$ of $\za'_1$,
$\za'_2$, $\za'_3$, $\za'_4$ are disjoint arcs.  The arcs $\zb'_i$ are
like those in NET map data item (3). Let $h'\co S^2\to S^2$ be a push
map relative to $\zb'_1$, $\zb'_2$, $\zb'_3$, $\zb'_4$ as in
Section~\ref{sec:prelim}.

We have that $f=h\circ g$, where $h$ is a push map and $g$ is the
Euclidean NET map determined by $\zF$ and the usual branched covering
map $\zp'_1\co \mathbb{R}^2\to S^2$ corresponding to $\zG'_1$.
So $f=\zj\circ (h'\circ g)$, where $\zj=h\circ h'^{-1}$ is an
orientation-preserving homeomorphism which stabilizes the postcritical
set $P_2$ of $f$.  Statement 2 of Lemma~\ref{lemma:induceup} with
$\zL=\zL'=\zL_2 =\mathbb{Z}^2$ easily implies that there exists an
affine isomorphism $\zJ\co \mathbb{R}^2\to \mathbb{R}^2$ such that
$\zJ(\zL_2)=\zL_2$ and $\zJ$ induces by means of the usual branched
covering map $\zp_2$ a homeomorphism from $S^2$ to $S^2$ which
is isotopic to $\zj$ rel $P_2$.  After modifying $f$ by an isotopy, we
may assume that $\zJ$ induces $\zj$.  Suppose that $\zJ(x)=Cx+d$.

Now we apply Theorems~\ref{thm:twisting1} and \ref{thm:twisting2}.  We
conclude that $f$ is Thurston equivalent to a NET map arising from NET
map data with the following form.  The lattice $\zL_2$ is preserved.
The orientation-preserving affine isomorphism has linear term $A=CA'$.
The associated green line segments have the form $\za_i=C\cdot
\za'_i$.  We take $\zl_1$ and $\zl_2$ to be the columns of $A$.
Because $\zl'_1$ is divisible by $m$ and $\zl'_2$ is divisible by $n$,
the corresponding statements are true for $\zl_1$ and $\zl_2$.

It easily follows that we now have a NET map presentation for our
original NET map and even that data item (4) is satisfied.  We obtain
our diagram from this as discussed in Section~\ref{sec:diagram2map}.

This proves Theorem~\ref{thm:pren}. 

\end{proof}

\section{An algorithm for computing NET map presentations}
\label{sec:algo}\nosubsections

The definition of a NET map is simple---it is easy to check whether or
not a Thurston map is a NET map.  The presentation of a NET map is
also easy to understand.  But finding a presentation of a given NET
map can be challenging.  The goal of this section is to provide an
algorithm which produces a presentation of a given NET map based on
topological knowledge of its action on the 2-sphere.

We begin with a discussion of Euclidean NET maps.  Let $g$ be a
Euclidean NET map.  Suppose that $g$ has a NET map presentation with
matrix $A$.  The translation term $b$ is irrelevant for this
discussion, and, of course, its green line segments are all trivial.

We wish to determine $A$ in terms of the action of $g$ on the 2-sphere
$S^2$.  Let $\zl_1$ and $\zl_2$ be the columns of $A$, as usual.  Let
$e_1$ and $e_2$ be the standard basis vectors of $\mathbb{Z}^2$.
Since $A$ expresses $\zl_1$ and $\zl_2$ in terms of $e_1$ and $e_2$,
the matrix $A^{-1}$ expresses $e_1$ and $e_2$ in terms of $\zl_1$ and
$\zl_2$.

Our next goal is to interpret the last statement in terms of simple
closed curves in $S^2$.  We have a branched covering map $\zp_1\co
\mathbb{R}^2\to S^2$ as usual taking 0, $\zl_1$, $\zl_2$,
$\zl_1+\zl_2$ onto the postcritical set $P_g$ of $g$.  Now we compute
slopes of simple closed curves in $S^2\setminus P_g$ using the ordered
basis $(\zl_1,\zl_2)$ rather than $(e_1,e_2)$.  Hence a simple closed
curve in $S^2\setminus P_g$ with core arc being the image under
$\zp_1$ of the line segment joining 0 and $\zl_1$ has slope 0.
Similarly, a simple closed curve in $S^2\setminus P_g$ with core arc
being the image under $\zp_1$ of the line segment joining 0 and
$\zl_2$ has slope $\infty$.

Let $\zg\subseteq S^2\setminus P_g$ be a simple closed curve with
slope 0.  The slopes of the connected components of $g^{-1}(\zg)$ are
equal.  Suppose that this slope is $\frac{p}{q}$, where $p$ and $q$
are relatively prime integers.  This implies that the vector $q
\zl_1+p \zl_2$ has the same direction as $\pm e_1$.  See
Figure~\ref{fig:computea}, where $p=-1$ and $q=2$.  Because $q \zl_1+p
\zl_2$ has integer coordinates, there is a positive integer $d$ such
that $\pm de_1=q \zl_1+p \zl_2$.  As in Theorem 4.1 of \cite{cfpp},
the integer $d$ is the degree with which $g$ maps every connected
component of $g^{-1}(\zg)$ to $\zg$.  Similarly, if $\zd$ is a simple
closed curve in $S^2\setminus P_g$ with slope $\infty$, then $\pm
ee_2=s \zl_1+r \zl_2$, where every connected component of
$g^{-1}(\zd)$ has slope $\frac{r}{s}$ and $g$ maps each of these
connected components to $\zd$ with degree $e$.  If necessary, we multiply one of the
columns of the matrix $\left[\begin{smallmatrix}q/d & s/e \\ p/d & r/e
\end{smallmatrix}\right]$ by $-1$ so that the resulting determinant is
positive.  Because the determinant of $A$ is positive, the resulting
matrix is $\pm A^{-1}$.  Thus in this way we are able to compute $A$
up to multiplication by $\pm 1$, which is what we need.

  \begin{figure}
\centerline{\includegraphics{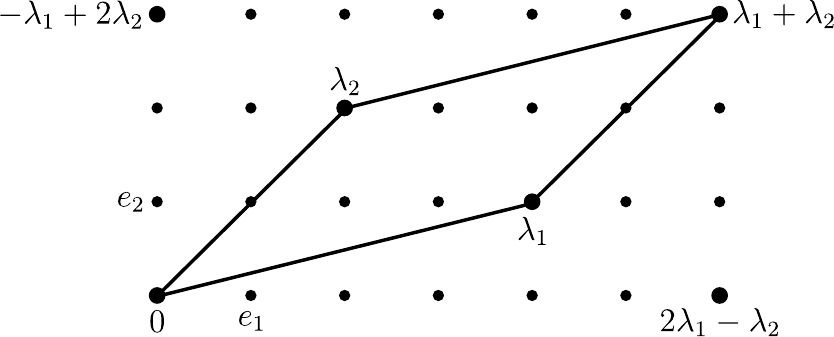}}  \caption{Computing a
matrix of a Euclidean NET map }
\label{fig:computea}
  \end{figure}
 
For example, for Figure~\ref{fig:computea}, we have that $p=-1$,
$q=2$ and $d=6$.  Similarly, $r=2$, $s=-1$ and $e=3$.  Checking
positivity of the determinant, we find that we may take
  \begin{equation*}
A=\begin{bmatrix}\frac{q}{d} & \frac{s}{e} \\ 
\frac{p}{d} & \frac{r}{e}\end{bmatrix}^{-1}=
\begin{bmatrix}\frac{1}{3} & \frac{-1}{3} \\ 
\frac{-1}{6} & \frac{2}{3} \end{bmatrix}^{-1}=
6\begin{bmatrix}\frac{2}{3} & \frac{1}{3} \\
\frac{1}{6} & \frac{1}{3} \end{bmatrix}=
\begin{bmatrix}4 & 2 \\ 1 & 2 \end{bmatrix},
  \end{equation*}
in agreement with the facts that $\zl_1=(4,1)$ and $\zl_2=(2,2)$.

Before describing the algorithm, we make a remark regarding notation.
The above proof of Theorem~\ref{thm:pren} begins by defining objects
denoted by primed characters.  In general, these objects must be
modified by a linear transformation to obtain the desired objects
denoted by unprimed characters.  In practice this linear
transformation is unnecessary.  So in our description of the
algorithm, we reverse the notation; we begin with unprimed characters
and end with primed characters.

We are now ready to present an algorithm which finds a presentation
for a general NET map $f$.  Here it is.
\medskip

\noindent
\textbf{Step 1. Identify the postcritical set $P_2$ of $f$.}
\medskip

\noindent
\textbf{Step 2. Identify the set $P_1$ of four points in $f^{-1}(P_2)$
which are not critical points.}
\medskip

Lemma 1.3 of \cite{cfpp} shows that $P_1$ exists.
\medskip

\noindent
\textbf{Step 3. Construct four disjoint (green) arcs $\zb_1$, $\zb_2$,
$\zb_3$, $\zb_4$ in $S^2$ each with one endpoint in $P_1$ and one
endpoint in $P_2$.}
\medskip

These arcs determine a push map $h$ up to isotopy rel $P_1$ such that
$h(P_1)=P_2$.  Set $g=h^{-1}\circ f$.  Then, as in Theorem 2.1 of
\cite{cfpp}, $g$ is a Euclidean NET map with postcritical set $P_1$.
\medskip

\noindent
\textbf{Step 4. Construct a simple closed curve in $S^2$ containing
$P_1$ which meets every $\zb_i$ in at most its endpoints.  Label the
points of $P_1$ with labels 0, $\zl_1$, $\zl_1+\zl_2$, $\zl_2$ in
cyclic order around the curve. This curve together with this labeling
of these four points determines a topological quadrilateral, $Q_1$, so
that the orientation of the labeled points is counterclockwise
relative to $Q_1$.}
\medskip

  \begin{figure} 
\centerline{\includegraphics{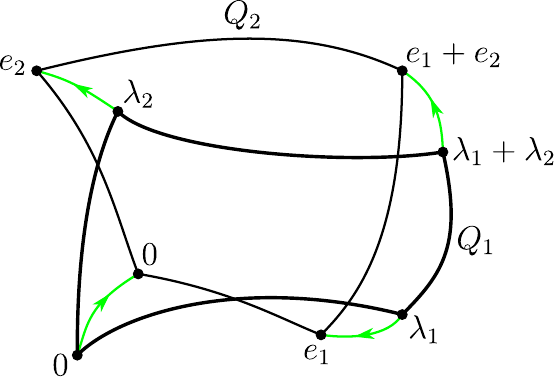}}\caption{The quadrilaterals 
$Q_1$ and $Q_2$}
\label{fig:q1q2}
  \end{figure}

See Figure~\ref{fig:q1q2}, where $\zb_1$, $\zb_2$, $\zb_3$, $\zb_4$
are drawn as green arcs.  Arrows on them indicate push directions.
One of the cyclic labelings of $Q_1$ is shown.  (The elements of $P_1$
will be the images of lattice points 0, $\zl_1$, $\zl_1+\zl_2$ and
$\zl_2$ in $\mathbb{R}^2$ under the usual branched covering map
$\zp_1\co \mathbb{R}^2\to S^2$.  We view $Q_1$ as the image of
the parallelogram in Figure~\ref{fig:computea} under $\zp_1$.)
\medskip

\noindent
\textbf{Step 5. (Optional) Construct $g^{-1}(\partial Q_1)$ up to
isotopy rel $P_1$.}
\medskip

Although this step is optional, it is useful when performing Step 6
and when constructing the covering map $\zp_1$ between Steps 8
and 9.  We may view $g^{-1}(\partial Q_1)$ as a 1-complex in $S^2$
with vertex set $g^{-1}(P_1)$.  The elements of $P_1$ are vertices of
$g^{-1}(\partial Q_1)$ with valence 2.  All other vertices have
valence 4.

Now we use $Q_1$ to define slopes of simple closed curves in
$S^2\setminus P_1$.  The edge of $Q_1$ with labels 0 and $\zl_1$
corresponds to slope 0, and the edge of $Q_1$ with labels 0 and
$\zl_2$ corresponds to slope $\infty$. (We are computing slopes here
using the ordered basis $(\zl_1,\zl_2)$ of $\mathbb{R}^2$, not the
standard ordered basis $(e_1,e_2)$.)  Let $\zg$ and $\zd$ be simple
closed curves in $S^2\setminus P_1$ with slopes 0 and $\infty$,
respectively.
\medskip

\noindent
\textbf{Step 6. Compute the slope $\frac{p}{q}$ in reduced form of one
component of $g^{-1}(\zg)$, and compute the degree $d$ with which $g$
maps this component to $\zg$.  Compute the slope $\frac{r}{s}$ in
reduced form of one component of $g^{-1}(\zd)$, and compute the degree
$e$ with which $g$ maps this component to $\zd$.}
\medskip

The 1-complex $g^{-1}(\partial Q_1)$ should be helpful here. 
\medskip

\noindent
\textbf{Step 7. Multiply one column of the matrix
$\left[\begin{smallmatrix}q/d & s/e \\ p/d & r/e
\end{smallmatrix}\right]$ by $-1$ if necessary so that the result has
positive determinant.  Compute the inverse $A$ of this matrix.}
\medskip

The discussion at the beginning of this section shows that the matrix
$A$ is a matrix of integers which is a presentation matrix for $g$.
\medskip

\noindent
\textbf{Step 8. Let $x$ be the element of $P_1$ with label 0, and determine
the label $b$ of $g(x)$.}
\medskip

Let $\zl_1$ and $\zl_2$ be the columns of $A$.  We have an affine
automorphism $\zF(x)=Ax+b$ of $\mathbb{R}^2$.  Now we define our usual
branched covering map $\zp_1\co \mathbb{R}^2\to S^2$ so that
$\zp_1(\zl)$ is the element of $P_1$ with label $\zl$ for
$\zl\in \{0,\zl_1,\zl_2,\zl_1+\zl_2\}$ and $\zp_1$ takes the
1-skeleton of the standard tiling of $\mathbb{R}^2$ by squares to
$g^{-1}(\partial Q_1)$, at least up to isotopy rel $\mathbb{Z}^2$.
The lifts of $Q_1$ to $\mathbb{R}^2$ yield fundamental domains for
$\zL_1$ and the lifts of $g^{-1}(Q_1)$ yield fundamental domains for
$\zL_2$.  The parallelogram $F_1$ in $\mathbb{R}^2$ with corners 0,
$2\zl_1$, $\zl_2$ and $2\zl_1+\zl_2$ is a fundamental domain for the
usual group $\zG_1$.
\medskip

\noindent 
\textbf{Step 9. Construct line segments $\za_i$ in $F_1$ such that the
arcs $\zp_1(\za_i)$ form a set of four disjoint arcs, each with
one endpoint in $P_1$ and one endpoint in $P_2$.}
\medskip

If a homotopy of $S^2$ rel $P_1\cup P_2$ takes the union of the arcs
$\zp_1(\za_i)$ to the union of the arcs $\zb_i$, then we are
done: we have an affine isomorphism $\zF(x)=Ax+b$, a parallelogram
which is a fundamental domain $F_1$ for $\zG_1$ and appropriate green
line segments $\za_i$.

Otherwise, we continue.  What follows is essentially an algorithm for
computing a NET map presentation of a twist of a NET map with known
presentation.  Let $k\co S^2\to S^2$ be a push map relative to the
images of the $\za_i$'s in $S^2$ such that $k(P_1)=P_2$.  Then
$f=h\circ g=\zj\circ (k\circ g)$, where $\zj=h\circ k^{-1}$ is an
orientation-preserving homeomorphism which stabilizes $P_2$.  We have
a presentation of the desired form for the NET map $k\circ g$.  We
wish to transform this presentation to a presentation for $f$, a
modular group twist of $k\circ g$.  This leads us to seek an affine
isomorphism $\zJ\co \mathbb{R}^2\to \mathbb{R}^2$ which induces $\zj$
relative to the branched covering map $\zp_2=f\circ \zp_1$.  For this
we work with simple closed curves in $S^2\setminus P_2$, calculating
their slopes using a quadrilateral $Q_2$ isotopic to $h(Q_1)$ rel
$P_2$ just as we used $Q_1$ to calculate slopes of simple closed
curves in $S^2\setminus P_1$.  See Figure~\ref{fig:q1q2}.  Now we take
$\zg$ and $\zd$ to be simple closed curves in $S^2\setminus P_2$ with
slopes 0 and $\infty$, respectively.
\medskip

\noindent
\textbf{Step 10. Compute the slopes $\frac{p}{q}$ and $\frac{r}{s}$ in
reduced form of $\zj(\zg)$ and $\zj(\zd)$}.
\medskip

\noindent
\textbf{Step 11. Multiply one column of the matrix
$\left[\begin{smallmatrix}q & s \\ p & r \end{smallmatrix}\right]$ by
$-1$ if necessary so that the result has determinant 1.  Let $C$ be
the resulting matrix.}
\medskip

\noindent
\textbf{Step 12. Label the vertices of $Q_2$ with labels 0, $e_1$,
$e_1+e_2$, $e_2$ consistently with the labels of $Q_1$.  Let $x$ be
the element of $P_2$ with label 0, and let $d$ be the label of
$\zj(x)$.}
\medskip

Then the affine automorphism $\zJ(x)=Cx+d$ of $\mathbb{R}^2$ induces
$\zj$ relative to $\zp_2$.  Theorems~\ref{thm:twisting1} and
\ref{thm:twisting2} now show how to transform the presentation diagram
of $k\circ g$ to obtain a presentation diagram for $f$.
\medskip

\noindent 
\textbf{Step 13. Compute the matrix $A'=CA$, vector $b'\in \{0,C
\zl_1,C \zl_2,C(\zl_1+\zl_2)\}$ congruent to $CAd+Cb$ modulo $2\left<C
\zl_1,C \zl_2\right>$, fundamental domain $F'_1=CF_1$ and line
segments $\za'_i=C\za_i$.}
\medskip

These primed quantities provide our presentation for $f$.

\begin{ex}\label{ex:rabbit} We illustrate the above algorithm by
deriving the NET map presentation for the rabbit $f(z)=z^2+c_R$ in
Figure~\ref{fig:rabbitpren}.  Figure~\ref{fig:rabbitfsr} is a diagram
of a finite subdivision rule for $f$.  The initial cell structure on
$S^2$ is on the right.  Its vertices correspond to the postcritical
points of $f$.  The vertex corresponding to $\infty$ is at infinity.
The vertices of the triangle on the right correspond to 0, $c_R$ and
$c_R^2+c_R$ in counterclockwise order starting at the bottom.  An edge
on the left with label $n$ maps to the edge on the right with label
$n$.
\medskip

  \begin{figure}
\centerline{\includegraphics{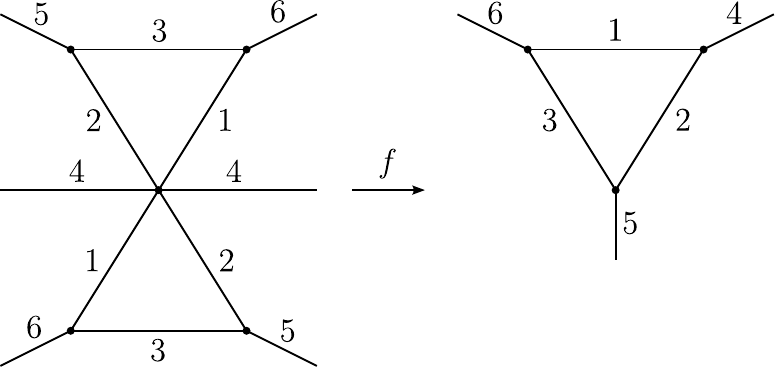}}
\caption{A finite subdivision rule for the rabbit $f$}
\label{fig:rabbitfsr}
  \end{figure}

\noindent
\textbf{Step 1.} $P_2=\{0,c_R,c_R^2+c_R,\infty\}$
\medskip

\noindent
\textbf{Step 2.} $P_1=\{\pm c_R,\pm (c_R^2+c_R)\}$
\medskip

\noindent \textbf{Step 3.} We choose green arcs $\zb_1$, $\zb_2$,
$\zb_3$, $\zb_4$ as in the right side of Figure~\ref{fig:rabbitg}.
Two of these arcs are drawn as green arcs with arrows indicating the
push direction.  There are trivial arcs at the vertices labeled
$\zl_2$ and $\zl_1+\zl_2$.
\medskip

  \begin{figure}
\centerline{\includegraphics{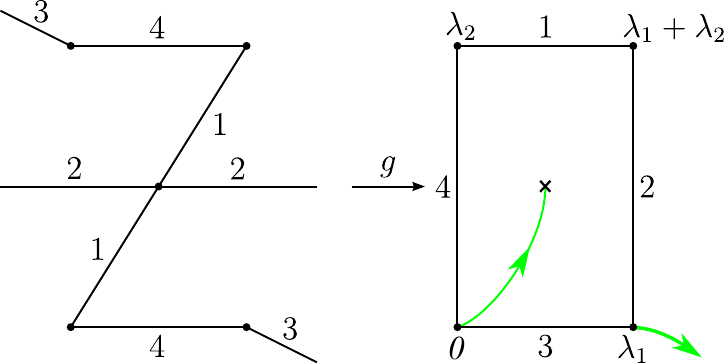}}
\caption{$\partial Q_1$ and its pullback under $g$}
\label{fig:rabbitg}
  \end{figure}

\noindent
\textbf{Step 4.} The quadrilateral $Q_1$ is drawn in the right side of
Figure~\ref{fig:rabbitg}.  
\medskip

\noindent \textbf{Step 5.} The complex $g^{-1}(\partial Q_1)$ is drawn
in the left side of Figure~\ref{fig:rabbitg}.  An edge with label $n$
on the left maps by $g$ to the edge with label $n$ on the right.
\medskip

\noindent
\textbf{Step 6.} To compute $\frac{p}{q}$, we view the edge on the
right with label 1 as a core arc for $\zg$.  Pulling back this edge,
we see that $\frac{p}{q}=1$ and $d=2$.  To compute $\frac{r}{s}$, we
view the edge on the right with label 4 as a core arc for $\zd$.
Pulling back this edge, we see that $\frac{r}{s}=0$ and $e=1$.
\medskip

\noindent
\textbf{Step 7.} $A=\begin{bmatrix}\frac{1}{2} & -1 \\ \frac{1}{2} &
0\end{bmatrix}^{-1}=\begin{bmatrix}0 & 2 \\ -1 & 1 \end{bmatrix}$
\medskip

\noindent
\textbf{Step 8.} $b=\zl_2$
\medskip

\noindent
\textbf{Step 9.} We obtain the diagram in Figure~\ref{fig:rabbitpren}.

\end{ex}

\begin{ex}\label{ex:lodge} For a second example, we work through the
algorithm to find a NET map presentation for
  \begin{equation*}
f(z)=\frac{3z^2}{2z^3+1}.
  \end{equation*}
This is the main example in Lodge's thesis \cite{l}.  It also appears
in Section 4 of \cite{BEKP} and Example 3.2 of \cite{cfpp}.

Note that $f(\zw z)=\zw^2f(z)$, where $\zw=e^{2\zp i/3}$. We compute:
  \begin{equation*}
f'(z)=\frac{6z(2z^3+1)-3z^26z^2}{(2z^3+1)^2}=\frac{6z(1-z^3)}{(2z^3+1)^2}.
  \end{equation*}
The critical points of $f$ are 0, 1, $\zw$ and $\zw^2$.  They are all
simple.  The postcritical points of $f$ are 0, 1, $\zw$ and $\zw^2$.
So $f$ is indeed a NET map.  We have the following table of values.

\begin{center}
\begin{tabular}{c|cccccccc}
$z$ & $\infty$ & 0 & 1 & $\frac{-1}{2}$ & $\zw$ & $\frac{-\zw}{2}$ &
$\zw^2$ & $\frac{-\zw^2}{2}$ \\ \hline
$f(z)$& 0 & 0 & 1 & 1 & $\zw^2$ & $\zw^2$ &  $\zw$ & $\zw$\\ 
\end{tabular}
\end{center}
\medskip

We see that $f$ is decreasing on the intervals $(-\infty,0)$ and
$(1,\infty)$ and that $f$ is increasing on the interval $(0,1)$.  We
conclude that $f$ maps $[-\frac{1}{2},\infty]$ to $[0,1]$ in 3-to-1
fashion.  This proves that the subdivision rule presentation of $f$ in
Figure~\ref{fig:Lodgefsr} is correct.  The label of an edge on the
left is the label of its image edge on the right.

  \begin{figure}
\centerline{\scalebox{.9}{\includegraphics{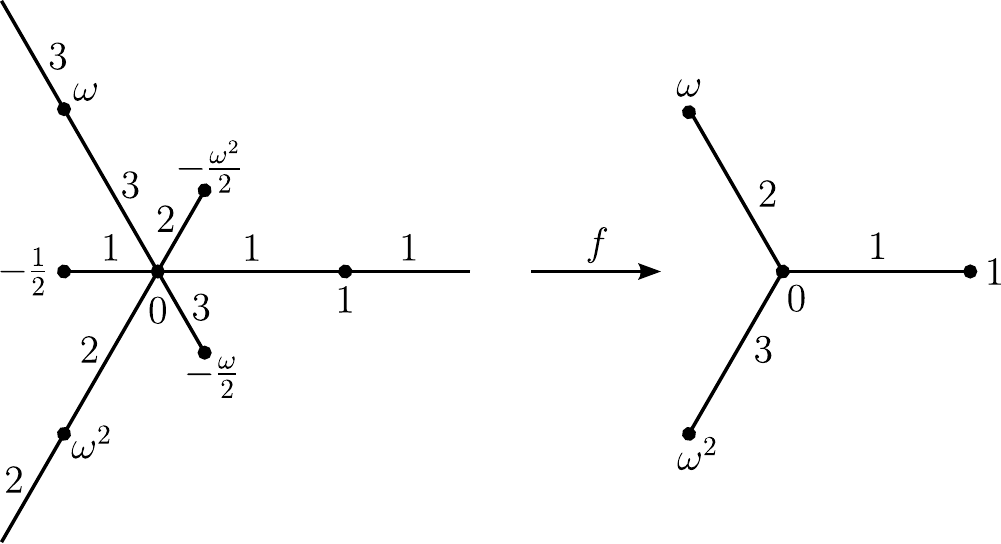}}} \caption{A finite
subdivision rule for $f$}
\label{fig:Lodgefsr}
  \end{figure}

We are prepared to apply the algorithm to find a NET map presentation
for $f$.
\medskip

\noindent
\textbf{Step 1.} $P_2=\{0, 1, \zw, \zw^2\}$
\medskip

\noindent
\textbf{Step 2.}
$P_1=\{-\frac{1}{2},-\frac{\zw}{2},-\frac{\zw^2}{2},\infty\}$
\medskip

\noindent
\textbf{Step 3.} We construct green arcs $\zb_1$, $\zb_2$, $\zb_3$,
$\zb_4$ in $S^2$ as in Figure~\ref{fig:Lodgeg}.  Arrows indicate push
directions.  In the right part of Figure~\ref{fig:Lodgeg}, elements of
$P_1$ are marked with dots and elements of $P_2$ are marked with
$\times $'s.
\medskip

  \begin{figure}
\centerline{\scalebox{.9}{\includegraphics{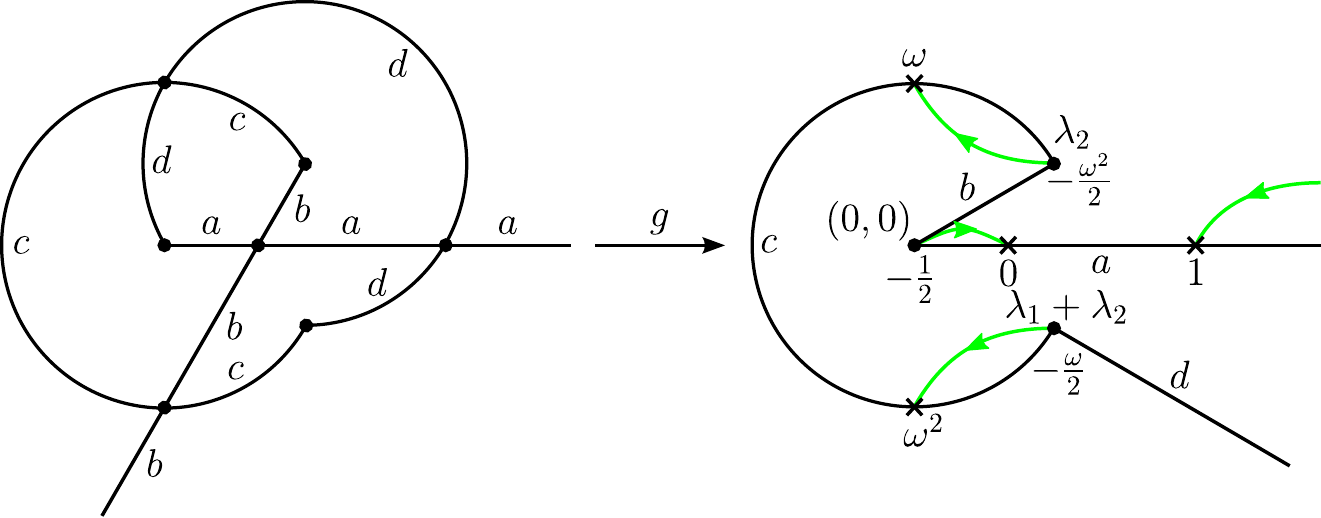}}} \caption{The
boundary of the quadrilateral $Q_1$ and its pullback under $g$}
\label{fig:Lodgeg}
  \end{figure}

\noindent
\textbf{Step 4.} We construct the quadrilateral $Q_1$ so that its
boundary is the union of the black line segments and circular arc in
the right side of Figure~\ref{fig:Lodgeg}.  We label the vertices of
$Q_1$ as shown in Figure~\ref{fig:Lodgeg}.  The label $\zl_1$ is at
$\infty$.
\medskip

\noindent 
\textbf{Step 5.} The complex $g^{-1}(\partial Q_1)$ appears
in the left side of Figure~\ref{fig:Lodgeg}, correct up to isotopy rel
$P_1\cup P_2$.  Its edges are labeled with letters.  The label of an
edge is the label of its image in $\partial Q_1$.
\medskip

\noindent
\textbf{Step 6.} The interval $[-1/2,\infty]$ has label $a$ in the
right side of Figure~\ref{fig:Lodgeg}.  It is a core arc for a simple
closed curve $\zg$ in $S^2\setminus P_1$ with slope 0.
Figure~\ref{fig:Lodgeg} shows that $g^{-1}(\zg)$ consists of one
connected component with slope 0.  So $\frac{p}{q}=0$ and $d=3$.  The
line segment in the right side of Figure~\ref{fig:Lodgeg} with label
$b$ is a core arc for a simple closed curve $\zd$ in $S^2\setminus
P_1$ with slope $\infty$.  Using Figure~\ref{fig:Lodgeg}, we find that
$g^{-1}(\zd)$ consists of one connected component with slope $-3$.  So
$\frac{r}{s}=-3$ and $e=3$.
\medskip

\noindent
\textbf{Step 7.} To get a matrix with positive determinant, we take
  \begin{equation*}
A=\begin{bmatrix}\frac{q}{d} & \frac{s}{e} \\ \frac{p}{d} &
\frac{r}{e} \end{bmatrix}^{-1}
=\begin{bmatrix}\frac{1}{3} & -\frac{1}{3} \\ 0 & 1 \end{bmatrix}^{-1}
=\begin{bmatrix}3 & 1 \\ 0 & 1 \end{bmatrix}.
  \end{equation*}
\medskip

\noindent
\textbf{Step 8.} We see that $b=\zl_1$.
\medskip

\noindent
\textbf{Step 9.} We obtain the NET map presentation diagram in
Figure~\ref{fig:LodgeDgm}. 
\medskip

  \begin{figure}
\centerline{\scalebox{.9}{\includegraphics{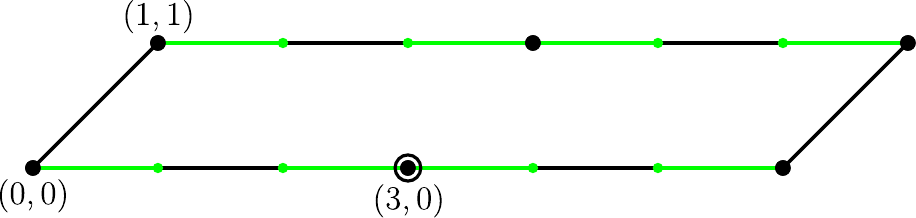}}} \caption{A
presentation diagram for $f$}
\label{fig:LodgeDgm}
  \end{figure}
\end{ex}

\end{document}